\documentclass{micls}
\usepackage{color,graphicx}
\usepackage{amssymb}
\usepackage{hyperref}
\def\newpic#1{%
\def\emline##1##2##3##4##5##6{%
\put(##1,##2){\special{em:point #1##3}}%
\put(##4,##5){\special{em:point #1##6}}%
\special{em:line #1##3,#1##6}}}
\newpic{}
\def\emline#1#2#3#4#5#6{%
\put(#1,#2){\special{em:moveto}}%
\put(#4,#5){\special{em:lineto}}}
\def\newpic#1{}

\newauthor{%
Italo J. Dejter and Oscar Tomaiconza}{%
I. J. Dejter and O. Tomaiconza}{%
{\rm University of Puerto Rico}\\
{\rm Rio Piedras, PR 00936-8377}}[%
oscar.tomaiconza@hotmail.com]

\title{Nonexistence of Efficient Dominating Sets in the Cayley Graphs Generated by Transposition Trees of Diameter 3}


\classnbr{05C69; 05C70; 05C12}

\begin{document}
\begin{abstract}
Let $d,n$ be positive integers such that $d<n$, and let $X^d_n$ be a Cayley graph generated by a transposition tree of diameter $d$. It is known that every $X^d_n$ with $d<3$ splits into efficient dominating sets. The main result of this paper is that $X^3_n$ does not have efficient dominating sets.
\end{abstract}

\section{INTRODUCTION}\label{s1}

Cayley graphs are very important for their  useful applications (cf. \cite{Kelarev1}), including to automata theory (cf. \cite{Kelarev2,Kelarev3}), interconnection networks (cf. \cite{akers, arumuga1,DS,ganesan,ga2}) and coding theory (cf. \cite{buet,DS}).

Let $0<d<n$ in $\mathbb{Z}$, and let $X^d_n$ be a Cayley graph generated by a transposition tree of diameter $d$. In \cite{DS}, it was shown that every $X^d_n$ with $d<3$ splits into efficient dominating sets. In the present work, the following result is proved.

\begin{theorem}\label{main}
Let $3<n$. Then no $X^3_n$ has efficient dominating sets.
\end{theorem}

The rest of this section is devoted to preliminaries and a plan of our proof of Theorem~\ref{main}.
Let $0<n\in\mathbb{Z}$ and let $I_n=\{1,2,\ldots,n\}$. Let $S_n$ be the group of permutations $\sigma={1\cdots\cdots n\choose\sigma_1\cdots\sigma_n}:I_n\rightarrow I_n$ with $\sigma(i)=\sigma_i$ for every $i\in I_n$ and $\{\sigma_1,\ldots,\sigma_n\}=I_n$. We write $\sigma=\sigma_1\cdots\sigma_n$. Thus, $e=12\cdots n$ means the identity of $S_n$.
Let ${\mathcal C}\subseteq S_n\setminus\{e\}$ satisfy $\sigma\in {\mathcal C}\Leftrightarrow\sigma^{-1}\in {\mathcal C}$.
The {\it Cayley graph} $X=X(S_n,{\mathcal C})$ of $S_n$ with {\it connection set} ${\mathcal C}$ is the graph $X=(S_n,E)$ with $gh\in E\Leftrightarrow h=\sigma g$, where $\sigma=hg^{-1}\in {\mathcal C}$. Here, if $\sigma=\sigma^{-1}$, we say that $gh\in E$ has {\it color} $\sigma$.

In \cite{GR}, Lemma 3.7.4 shows that $X$ is connected if and only if $\mathcal C$ is a generating set for $S_n$, and Lemma 3.10.1 shows that a set of transpositions $(ij)$ of $S_n$, with $i\ne j$ in $I_n$, generates $S_n$ if and only if the graph $\tau$ whose edges are of the form $ij$ is connected. We start Subsection~\ref{T} by considering such a graph  $\tau$.

\subsection{Transpositions, Domination and Packing}\label{T}
Let $\tau$ be a connected graph with vertex set $I_n$ and let ${\mathcal C}= {\mathcal C}_\tau$ be composed by the transpositions $\sigma=(ij)$, where $ij$ runs over the edges of $\tau$. Then $\sigma=\sigma^{-1}$ for each $\sigma\in{\mathcal C}_\tau$. This yields the graph $X(S_n,\tau)=X(S_n, {\mathcal C}_\tau)$ as an edge-colored graph via the {\it color set} $\mathcal C_\tau$ with a 1-factorization into the 1-factors $F_\sigma=F_{ij}$ of $\sigma$-colored edges. Here, $\tau$ is called the {\it transposition graph} of $X(S_n,\tau)$ \cite{ganesan,ga2}.

For domination and packing in Cayley graphs, the terminology of \cite{haynes} is used.
A {\it stable subset} $J\subseteq S_n$ (i.e. a set of nonadjacent vertices) with each
vertex of $S_n\setminus J$ adjacent in the Cayley graph $X$ to just one vertex of $J$ is an {\it efficient dominating set} (or {\it E-set}) of $X$.
The $1$-{\it sphere} with {\it center} $g\in S_n$ is the subset $\{h\in S_n\,|\,\rho(g,h)\le 1\}$, where $\rho$ is the graph distance of $X$.
Every E-set in $X$ is the set of centers of the 1-spheres in a {\it perfect sphere packing} (as in \cite{JJ}, page 109) of $X$.
Let $X'$ be a proper subgraph of $X$ ($X'$ specified in Subsection~\ref{444}).
Let $\mathcal S$ be a perfect 1-sphere packing of $X'$.
The union of a 1-sphere of $\mathcal S$ with its neighbors in $S_n\setminus V(X')$ is an {\it $\mathcal S$-sphere}.
The union of two $1$-spheres centered at adjacent vertices $x,x'$ of $X$ is a {\it double-sphere} with {\it centers} $x,x'$.
A collection of pairwise disjoint $1$-spheres (resp., $\mathcal S$-spheres and dou\-ble-spheres) in $X$ is said to be a {\it 1-sphere packing} of $X$ (resp., a {\it special packing} of $X$, to be used in Section~\ref{s5}).
It may happen that $X$ has a packing $\mathcal T$ by $\mathcal S$-spheres, see Figure 1 below.

Given a packing $\mathcal S$ of 1-spheres in $X$ whose union has cardinality $\alpha|S_n|=\alpha n!$, ($0<\alpha\le 1$), the set $J$ of centers of the 1-spheres of $\mathcal S$ is an {\it$\alpha$-efficient dominating set} (or $\alpha${\it-E-set}) of $X$,
in which case we may denote (by abuse of notation) the induced subgraph $X[J]$ by $J$.
Note that a 1-E-set is an E-set, and viceversa.

\subsection{Transposition Trees of Diameter less than 3}\label{xyz}
Theorem 3.10.2 \cite{GR} implies that ${\mathcal C}_\tau$ is a minimal generating set for $S_n$ $\Leftrightarrow$ $\tau$ is a tree.
We take $\tau=\tau^{d_n}=\tau^d$ to be a diameter-$d$ tree and denote $X_n^d=X(S_n,\tau^d)$.
Let $\tau^{d_1}=\tau^0=(I_1,\emptyset)$. Let $\tau^{d_n}=K_{1,n-1}$ with $d_n=2$ if $n>2$ and $d_n=1$ if $n=2$. By assuming $1\in I_n=V(\tau^{d_n})$ of degree $n-1$,  $S_n=V(X_n^{d_n})$ splits into E-sets $\xi^1_i=i(I_n\setminus\{i\})$, ($i\in I_n$), formed by those $\sigma\in S_n$ with $\sigma_1=i$ \cite{arumuga1}. (For example, $\xi^1_1=1(2,\ldots,n)$, also written as $\xi^1_1=1(2\cdots n)$).
In this terms, \cite{DS} showed that if $n>1$ then for each $i\in I_n$, $X_n^{d_n}-\xi^1_i$ is the disjoint union of $n-1$ copies $\xi^j_i$ of $X_{n-1}^{d_{n-1}}$, where $\xi^j_i$ is induced by all $\sigma\in S_n$ with $\sigma_j=i$ and $j\in I_n\setminus\{1\}$. This is used in proving Theorem~\ref{main} as we indicate in Subsections~\ref{r10} and~\ref{444}.

\subsection{Transposition Trees of Diameter 3}\label{r10}
A diameter-3 tree $\tau^3$ has two vertices of degrees $r,t$ larger than 1 joined by an edge $\epsilon$. Then $n=r+t$. We write $\tau^3=\tau_{r,t}^3$ and take: {\bf(i)} $r$ and $r^*=r+1$ as the vertices of $\tau_{r,t}^3$ of degrees $r$ and $t$ so that $\epsilon=rr^*$; {\bf(ii)} $1,\ldots,r-1$ (resp., $r^*+1,\ldots,n$) as the neighbors of $r$ (resp., $r^*$) in $\tau_{r,t}^3$. (This vertex numbering is modified in Sections~\ref{s6}-\ref{s8}).
Edge pairs in $\tau_{r,t}^3$ induce copies of both: {\bf(A)} the disjoint union $2K_2=2P_2$ of two paths of length 1; {\bf(B)} the path $P_3$ of length 2. Using two-color alternation in $X_{r,t}^3=X(S_n,\tau_{r,t}^3)$, the edge pairs (A) (resp., (B)) determine 4-cycles (resp., 6-cycles).
The subgraphs of $X_{r,t}^3$ induced by the ${n\choose r}$ cosets of $S_r\times S_t$ in $S_n$ are the components of the subgraph $X_{r,t}^3\setminus F_\epsilon$ of $X_{r,t}^3$, see Subsection~\ref{T}. These components are copies of a cartesian product $\Pi_r^t=X_r^{d_r}\square X_t^{d_t}$ with: {\bf(a)} $d_r=d_t=2$, if min$(r,t)>2$;
{\bf(b)} $d_r=2=d_t+1$, if $r>t=2$; {\bf(c)} $d_r=d_t=1$, if $r=t=2$.

If an $\alpha$-E-set $J$ of $X_{r,t}^3$ is
equivalent in all copies of $\Pi_r^t$ of $X_{r,t}^3$\,, than
both $J$ and its associated 1-sphere packing are said to be {\it uniform}. There is no uniform $\alpha$-E-set in $X_{2,2}^3$, see Figure 1 below.
Theorem~\ref{t2} will show that if $4<n=r+t$, then
uniform $\alpha$-E-sets of $X_{r,t}^3$ have $\alpha\le\frac{n}{rt}<1$.
Theorem~\ref{t5} and Corollary~\ref{r6} will certify that such an upper bound $\frac{n}{rt}$ can only be attained by uniform $\alpha$-E-sets that
intersect each copy of $\Pi_r^t$ in a product $J'\times J''$ of E-sets $J'\subset X_r^{d_r}$ and $J''\subset X_t^{d_t}$.
Then, all $\alpha$-E-sets in the graphs $X_{r,t}^3$ happen with $\alpha<1$ and Theorem~\ref{main} follows.
Our plan of proof is complemented in Subsection~\ref{444}.

\begin{figure}[htp]
\hspace*{13mm}
\includegraphics[scale=0.3]{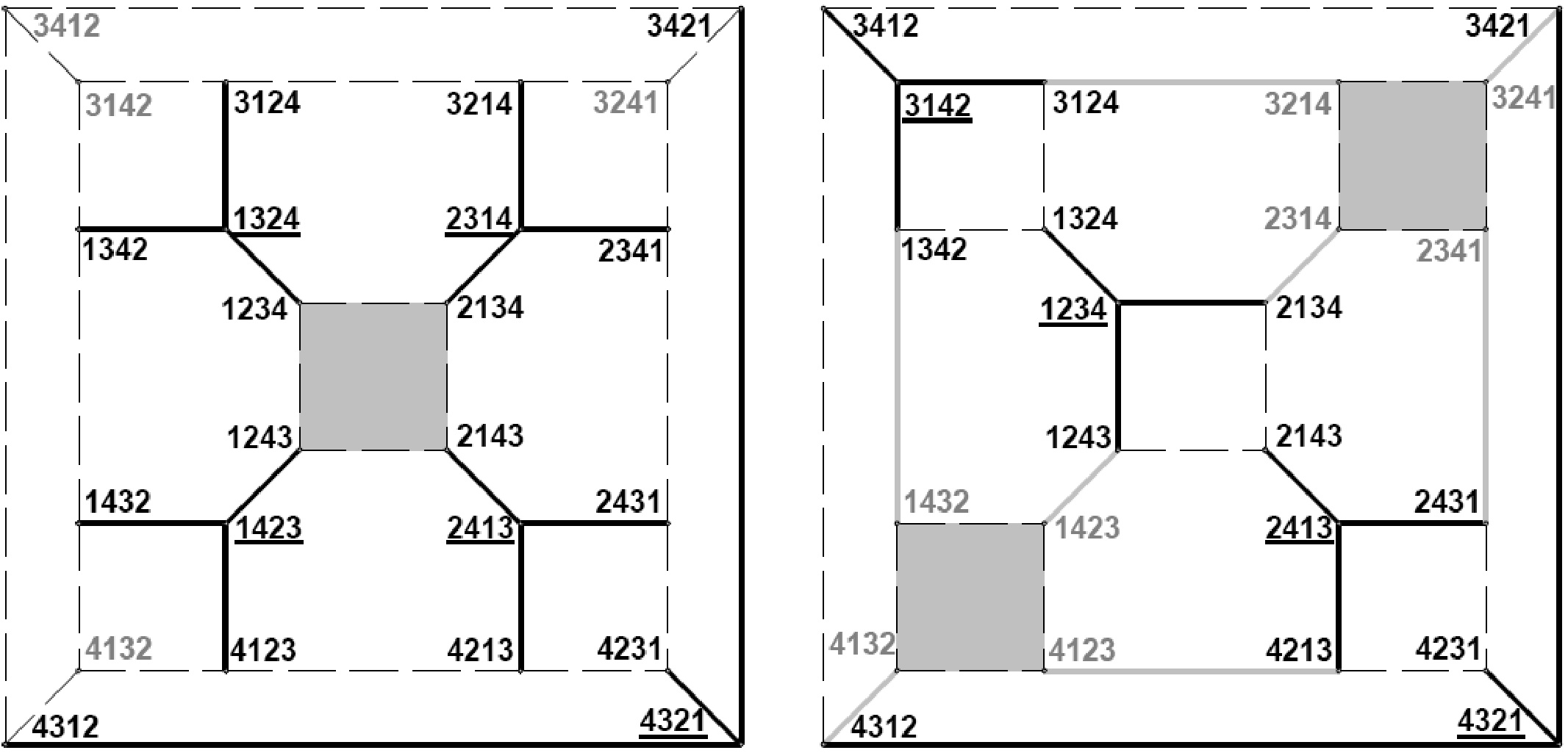}
\caption{Representations of a $(5/6)$- and a $(2/3)$-E-set of $X_{2,2}^3$}
\end{figure}

Every $\alpha$-E-set in $X_{2,2}^3$ avoids at least one of the six copies of $\Pi_2^2$ in $X_{2,2}^3$. See the two instances of $\alpha$-E-sets in $X_{2,2}^3$ in Figure 1, with each avoided copy of $\Pi_2^2$ bounding a solid-gray square. On the left, the edges incident to a $(5/6)$-E-set are in thick trace.
(In expressing $n$-tuples in $S_n$, commas and parentheses are ignored).
On the right, (to be compared with the construction in Section~\ref{s5} and initiating the inductive construction of Section~\ref{s6}), a 1-sphere packing $\mathcal S$ of $X_{2,2}^3$ is shown that covers $16=(2/3)4!$ vertices, with underlined black 1-sphere centers. The 1-spheres of $\mathcal S$, forming a $(2/3)$-E-set, induce the edges in thick black trace. Of the other edges, those colored $(23)=(\epsilon)$, induced by the $\mathcal S$-spheres, forming a $\mathcal T$ as in Subsection~\ref{T}, are in thick light-gray. The eight vertices in the $\mathcal S$-spheres of $\mathcal T$ not in the 1-spheres of $\mathcal S$ are light-gray (in contrast with the remaining vertices, in black) and span two 4-cycles bounding solid gray squares as cited above.

\subsection{Largest Cayley Subgraph with an E-set}\label{444}
To obtain Theorem~\ref{t5}, we follow the following development in Sections~\ref{s5}-\ref{s8}. Let $r=t>2$. In each copy of $\Pi_r^t$ (Subsection~\ref{r10})
a partition of $S_r=V(X_r^{d_r})$ into E-sets (Subsection~\ref{xyz}) is combined by concatenation with a corresponding partition of the subgroup $A_t=V(X_t^{d_t}[A_t])$ of index 2 in $S_t$. Now, a connected subgraph $X'=X'_{r,t}$ induced by $2^r$ of the $n\choose r$ copies of $\Pi_r^t$ in $X_{r,t}^3$ has an E-set $J$. Here, $X'$ is the largest subgraph of $X_{r,t}^3$ with a perfect 1-sphere packing. Also, $V(X')$ is a subgroup of $S_n$ containing $J$ as a subgroup.
Theorem~\ref{t5} implies that
$J$, whose associated 1-sphere packing has maximum {\it localized packing density} (Section~\ref{s5} and following), cannot be extended to an E-set of $X_{r,t}^3$.
Moreover, $J$
extends to a maximum nonuniform $\alpha$-E-set of $X_{r,t}^3$ with largest $\alpha>\frac{n}{r^2}$ such that $\alpha<1$. Corollary~\ref{r6} allows to extend this case of $X_{r,r}^3$ to the case of $X_{r,t}^3$ ($r>t>2$), via puncturing restriction. This allows the completion of the proof of Theorem~\ref{t2}, and thus that of Theorem~\ref{main}.

\begin{rem} A conjecture in \cite{DS} says that no E-set of $X_n^d$ exists if $d>2$. Remark 1 \cite{buet} says that a proof of this conjecture as ``Theorem 5'' \cite{DS} fails. This can be corrected for $d>2$ by restricting to either $n=4$ or $n$ a prime $n>4$, proved in \cite{buet} for path graphs $\tau^d$. It can be  proved for any tree $\tau^d$ using   \cite{DS} Lemma 6 that generalizes the decomposition of $X_{r,t}^3\setminus F_\epsilon$ in Subsection~\ref{r10}.\end{rem}

\section{JOHNSON GRAPHS}\label{s2}

Let $2<r<n-1$ in $\mathbb{Z}$. Let $\Gamma^r_n=(V,E)$ be the edge-colored graph with $V=\{r\mbox{-subsets of }I_n\}$ and $tu\in E$ $\Leftrightarrow$ $t\cap u$ is an $(r-1)$-subset, said to be the {\it color} of $tu$. Note that $\Gamma^r_n$ is the Johnson graph $J(n,r,r-1)$ \cite{GR}.
%, which according to \cite {A} is Hamilton-connected, so it is Hamiltonian.
A subgraph $\Psi$ of $\Gamma^r_n$ is {\it exact} if: {\bf(a)} each two of its edges incident to a common vertex have the $(r-1)$-subsets representing their colors sharing exactly $r-2$ elements of $I_n$, and {\bf(b)} the vertices $u,v,w$ of
each $P_3=uvw$ in $\Psi$ involve $r+2$ elements of $I_n$, that is: $|u\cup v\cup w|=r+2$.
Exact spanning subgraphs $\Phi^r_n$ of $\Gamma^r_n$ are applied in Sections~\ref{s4}--\ref{spack} to packing 1-spheres into $X_{r,t}^3$.

An exact cycle in $\Gamma^3_5$ is $\psi_5=(345,234,123,512,451)$ (or in reverse, $\psi_5^{-1}=(321,432,543,154,215)$), where each triple $a_0a_1a_2$ acquires the element $a_0$ among those absent in the preceding triple and loses the element $a_2$ among those present in the following triple, with 3-strings taken cyclically mod 5. This is also expressed as a {\it condensed cycle} (or {\it CC}) of triples $\psi_5=(12345)$, (resp., $\psi_5^{-1}=(54321)$), whose successive composing triples yield corresponding successive terms of the original form of $\psi_5$, (resp., $\psi_5^{-1}$). We can take an exact $\Phi^3_5\in\{\{\psi_5,\psi_5'\},\{\psi_5^{-1},$

\noindent$\psi_5'^{-1}\},
\{\psi_5,\psi_5'^{-1}\},
\{\psi_5^{-1},\psi_5'\}\},$ where
\begin{eqnarray}\label{d0}
^{\psi_5\hspace*{2.3mm}=(345,\,234,\,123,\,512,\,451)=(12345),\;\psi_5'\hspace*{2.8mm}=(135,\,413,\,241,\,524,\,352)=(13524),}
_{\psi_5^{-1}=(321,\,432,\,543,\,154,\,215)=(54321),\;\psi_5'^{-1}=(142,\,314,\,531,\,253,\,425)=(53142),}
\end{eqnarray}
are expressed as cycles of triples in $\Gamma_5^3$ and as their respective CCs.

\section{APPLICATION TO SPHERE PACKING}\label{s4}

The exact 2-factor above combine with the decomposition of $X_{r,t}^3\setminus F_\epsilon$ into copies of $\Pi_r^t$ in Subsection~\ref{r10}. In preparation for Theorem~\ref{t2}, we provide an example.

\begin{figure}[htp]
\hspace*{1mm}
\includegraphics[scale=0.35]{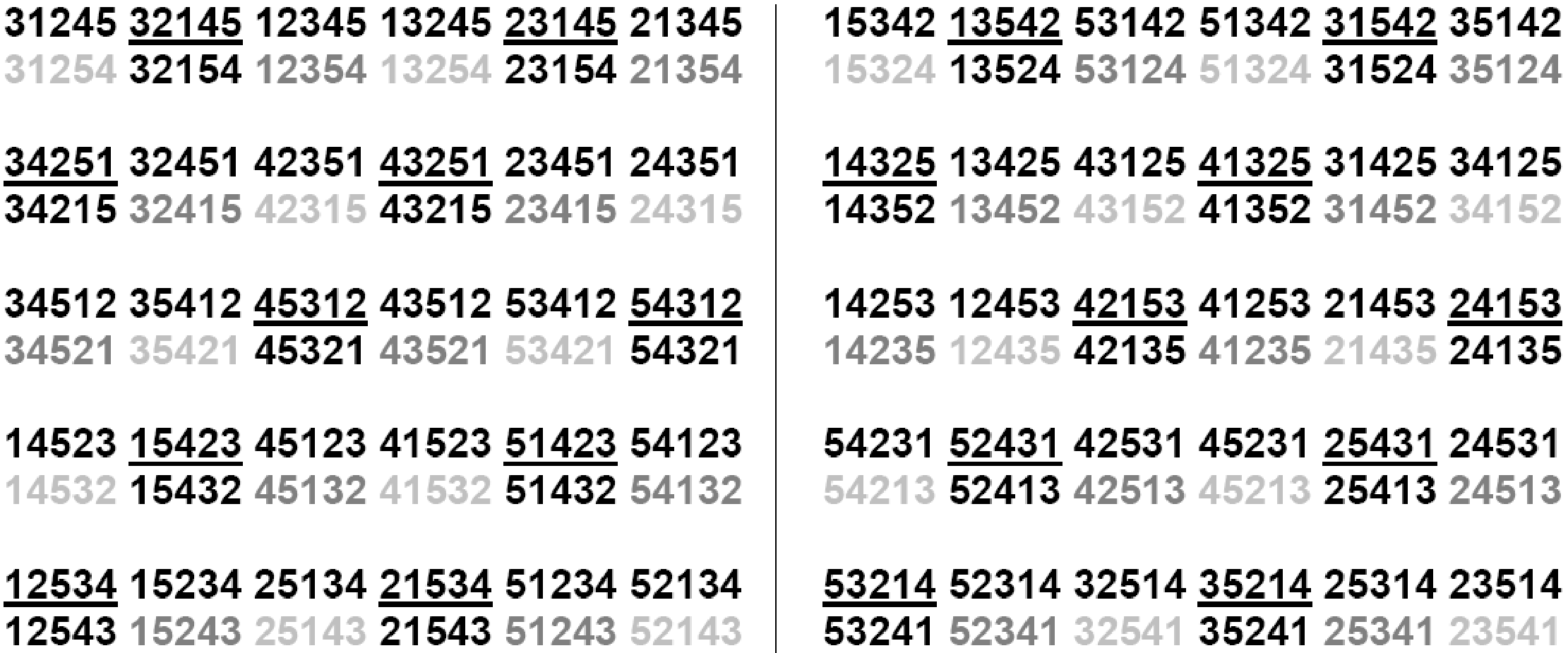}
\caption{A uniform $(5/6)$-E-set in $X_{3,2}^3$ via an exact $\Phi^3_5$}
\end{figure}

Note that $X_{3,2}^3\setminus F_\epsilon$, (where $(34)=(\epsilon)$), splits into ten copies of $\Pi_3^2=X_3^2\square X_2^1$. Each $2\times 6$ array in Figure 2
shows one such copy, composed by: {\bf(i)} two copies
of $X_3^2$ (shown as contiguous rows), i.e. two 6-cycles (obtained in the upper-left corner, by concatenating 45 or 54 to each entry of $(312,\xi^1_3,321,\xi^2_2,123,\xi^1_1,132,\xi^2_3,$\\
\noindent$231,\xi^1_2,213,\xi^2_1)$, with edges represented by the copies $\xi^i_j$ of $X_2^1$, using Subsection~\ref{xyz}); {\bf(ii)} six column-wise copies of $X_2^1$; {\bf(iii)} six 4-cycles given by  contiguous columns.
The five copies of $\Pi_3^2$ on the left of the figure
are in ordered correspondence with the terms of the 5-cycle
$\psi_5^{-1}=(321,432,543,154,215)$ in display (\ref{d0}):
the black vertices in each of the five copies of $\Pi_3^2$ determine
two 1-spheres with the two dark-gray
vertices in the subsequent copy of $\Pi_3^2$, where: {\bf(a)} the top copy
of $\Pi_3^2$ is taken to be subsequent to the bottom copy; {\bf(b)} the
center of each such 1-sphere is underlined; {\bf(c)} one of the two underlined
vertices in each copy of $\Pi_3^2$ starts with the triple given
by a corresponding term in $\psi_5^{-1}$; and {\bf(d)} the remaining vertices are
light-gray.
For example, a 1-sphere here is given by the underlined-black vertex 32145 (forming part of  the product $J=\xi^3_1\times\xi^4_4=(12)3\times 4(5)$ of E-sets in $\Pi_3^2=X_3^2\square X_2^1$) in the top copy of $\Pi_3^2$, its black neighbors 12345, 31245 and 32154 and the dark-gray vertex 32415 in the subsequent copy of $\Pi_3^2$.
Similarly, the five copies of $\Pi_3^2$ on the right of Figure 2 are linked to the 5-cycle $\psi_5'=(135,413,241,524,352)$. Now, the underlined vertices yield a $(5/6)$-E-set.

\section{CYCLIC ORDERED PARTITIONS}\label{xxx}

No exact 2-factor $\Phi_6^4$ exists. This is remedied in (B) below.
On the other hand, an exact 2-factor $\Phi_7^4$ is given by the CCs $\phi_1=(1234567)$, $\phi_2=(1357246)$, $\phi_3=(1473625)$, that we equalize to the respective {\it cyclic ordered partitions} (or COPs) $1114=\phi_1$, $2221=\phi_2$, $1213=\phi_3=\{1245=2514,2356=3627,\ldots,7134=1473\}$ of the integer 7 (associated with the successive difference triples 111, 222, 333 of quadruples) and by alternating the quadruples in the COPs
$$\begin{array}{l}
^{1123=\{1235,2346,3457,4561,5672,6713,7124\},}_{2113=\{1345,2456,3567,4671,
5712,6123,7234\},}\end{array}$$ into the exact CC $(1235,1345,4561,4671,7124,7234,3457,3567,6713,6123,2346,$

\noindent$2456,5672,5712).$
Note that $\Gamma_7^5$ has COPs $11113=\phi_1$, $11122=\phi_2$ and $11212=\phi_3$,
yielding an exact $\Phi_7^5$.

Exact spanning subgraphs of largest degree 3 in $\Gamma^r_n$ whose components are unicyclic caterpillars, (i.e. graphs for which the removal of its pendant vertices makes them cyclic) will be called {\it nests}. Then, a nest
leads to a uniform $\alpha$-E-set with $\alpha=\frac{n}{rt}$. For example: {\bf(A)}, the nest of $\Gamma^3_5$ formed by the CC $(12345)$ plus the edges $(132,135)$, $(423,421)$, $(354,352)$, $(415,413)$ and $(251,254)$ leads to a  uniform $(5/6)$-E-set; {\bf(B)}  In $\Gamma_6^4$, the  COPs 1113 and 1122 alternate into the exact 12-cycle
$$(1234,1235,2345,2346,3456,3451,4561,4562,5612,5613,6123,6124).$$ A nest is obtained by attaching edges with pendant vertices in the COP $1212=\{1245,2356,3461\}$, say edges
$(1235,1245)$, $(3451,3461)$ and $(5613,2356)$.
This leads to a uniform $(1/4)$-E-set in $X_{4,2}^3$;
an alternate nest of  $\Gamma^4_6$ is formed by the 5-cycles $$\begin{array}{l}(12345)|6=(1236,2346,3456,4516,5126)\\
(62413)|5=(6245,2415,4135,1365,3625)\end{array}$$ plus the edges
$(6245,1246)$, $(2415,1234)$, $(4135,5234)$,
$(1365,1346)$, $(3625,5123)$.

For $n>4$, exact non-spanning subgraphs of $\Gamma^r_n$ yield $\alpha<\frac{n}{rt}$. To exemplify this, we reselect the centers of disjoint 1-spheres in Figure 2 by taking all vertices in a copy of $\Pi_r^t$ as dark-gray and its neighbors via $F_\epsilon$ underlined-black, then setting as dark-gray enough vertices at distance 2 from underlined-black vertices, traversing $F_\epsilon$ to set underlined-black vertices in all copies of $\Pi_r^t$. One can select more than one copy of $\Pi_r^t$ to be completely dark-gray, e.g. those copies containing vertices 123456 and 654321 in $X_{3,3}^3$ and proceed as above until the twenty copies of $\Pi_r^t$ have underlined-black vertices, but the value of $\alpha$ in such cases is still less than $\frac{n}{rt}$.

\section{UNIFORM SPHERE PACKING}\label{spack}

Assume $4<n=r+t$, where $r,t\in\mathbb{Z}$. Then each copy $\Pi'$ of $\Pi_r^t=X^{d_r}_r\square X^{d_t}_t$ in $X_{r,t}^3$, where $d_r,d_t\in\{1,2\}$, has $r!t!$ vertices.
We use now from Sections~\ref{s5}-\ref{s8} below that covering a copy $\Pi'$ with 1-spheres of a packing $\mathcal S$ of $X_{r,t}^3$ prevents
$\mathcal S$ for being uniform. As a consequence, it arises from Sections~\ref{s4}-\ref{xxx} that uniform $\alpha$-E-sets $J$ in $X_{r,t}^3$ have $\alpha\le\frac{n}{rt}$, as their intersection with each $\Pi'$ is contained at most in a product of E-sets, guaranteeing $\alpha\le\frac{n}{rt}$. Moreover, if $\alpha=\frac{n}{rt}$ then each $\Pi'\cap J$ equals $J'\times J''$. Here, $J'$ and $J''$ are E-sets in $X^{d_r}_r$ and $X^{d_t}_t$ of the forms $\xi^r_i$ ($1\le i<r$) and $\xi^{r^*}_j$ ($r^*<j\le n$) respectively, (instead of $\xi^1_i=i(I_n\setminus\{i\})$ with $1<i\le n$, as in Subsection~\ref{xyz}).
Let $N[J'\times J'']$ be the union of the 1-spheres centered at the vertices of $J'\times J''$. Then $\Pi'-N[J'\times J'']$ is the disjoint union of $(r-1)(t-1)$ copies of $\Pi_{r-1}^{t-1}$. Also, each $\Pi'$ intersects $J$
in $(r-1)!(t-1)!$ vertices. These are the centers of pairwise disjoint 1-spheres, yielding a total of $(r-1)!(t-1)!n$ vertices in all those spheres. This way, $\frac{n}{rt}n!$ vertices of $X_{r,t}^3$ become covered by pairwise disjoint 1-spheres in $X_{r,t}^3$. This together with the outcome of Subsection~\ref{444}
yields a maximal imperfect uniform 1-sphere packing of $X_{r,t}^3$.
Such a packing ensures the nonexistence of E-sets of $X_{r,t}^3$ via the arguments of Theorem~\ref{t5} and Corollary~\ref{r6} below.

\begin{theorem}\label{t2}
Let $4<n=r+t$, ($r,t\in\mathbb{Z}$). Then, there are at most $\frac{n}{rt}n!$ vertices in the union of 1-spheres of an imperfect uniform 1-sphere packing of $X_{r,t}^3$. This ensures the nonexistence of E-sets of $X_{r,t}^3$.
\end{theorem}

\section{LOCALIZED PACKING DENSITY}\label{s5}

The techniques in this and following sections lead to maximum {\it localized packing density}, meaning the packing of as many 1-spheres as possible in a specific copy of $\Pi_r^t$ according to the decomposition of $X_{r,t}^3\setminus F_\epsilon$ in Subsection~\ref{r10}.

\begin{figure}[htp]
\hspace*{3.0mm}
\includegraphics[scale=0.36]{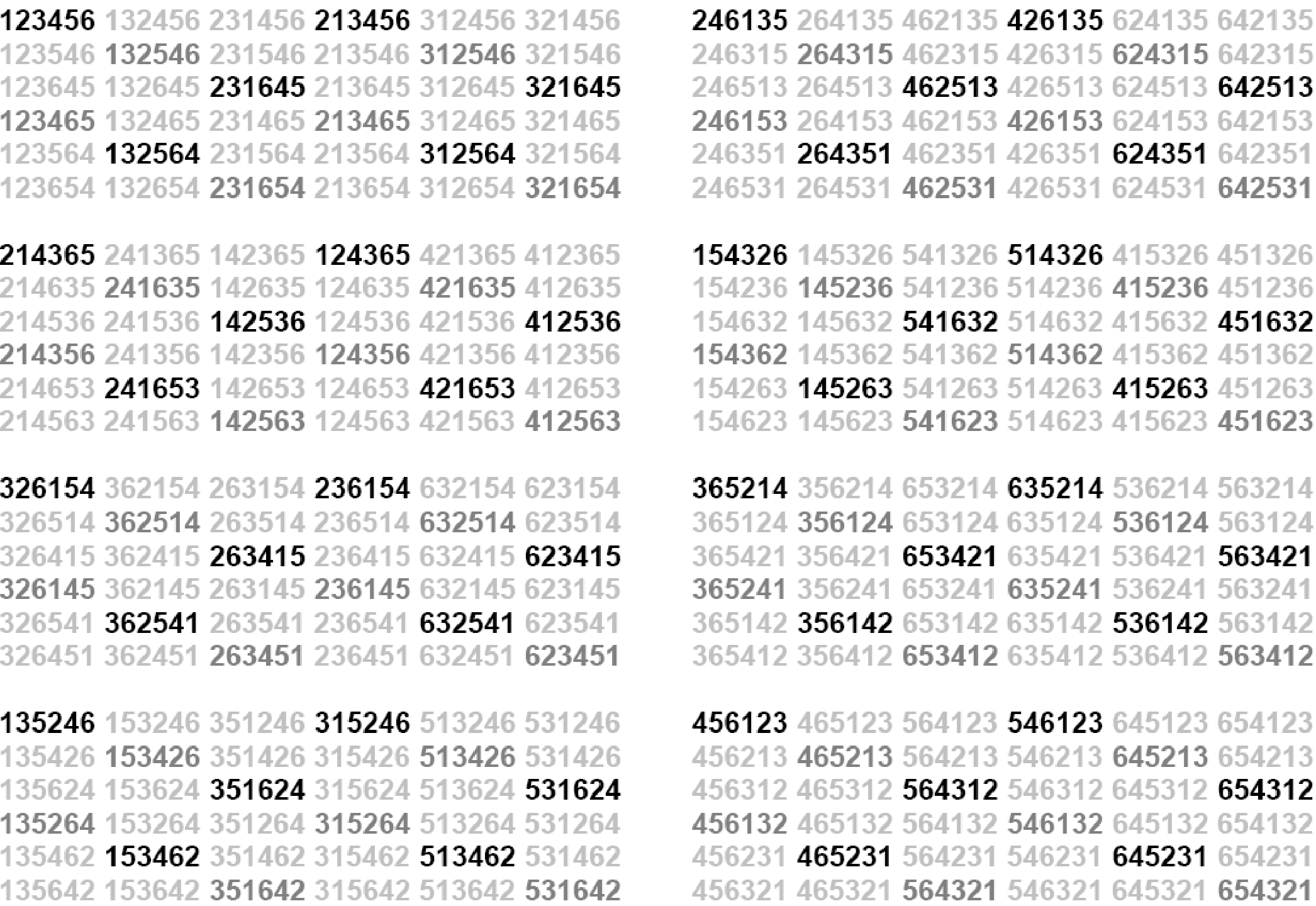}
\caption{Local maximum packing density in $X_{3,3}^3$}
\end{figure}

\begin{figure}[htp]
\hspace*{0.5cm}
\includegraphics[scale=0.28]{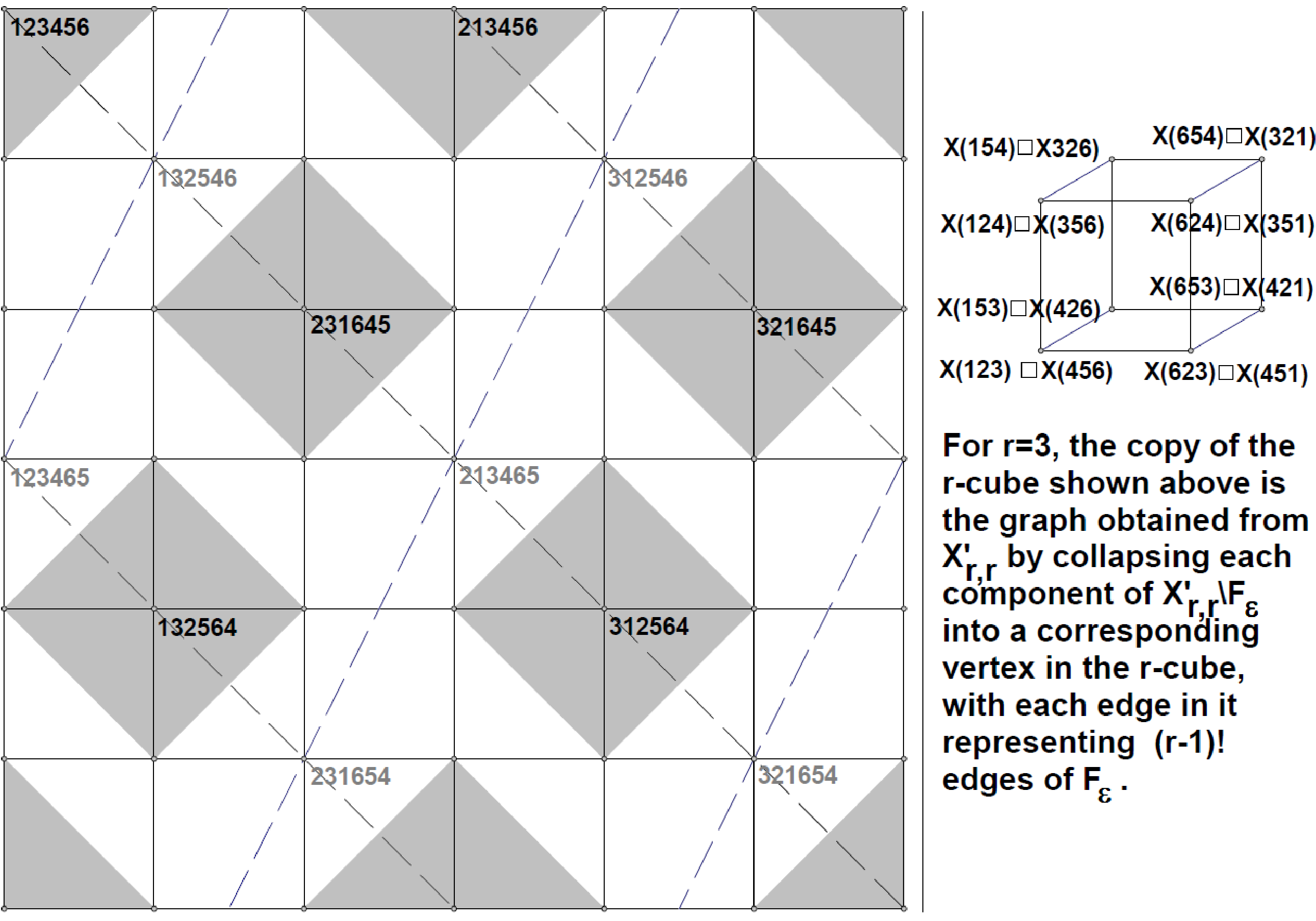}
\caption{Embedding of $X(123)\square X(456)$ in a torus and a representation of $X'_{3,3}$}
\end{figure}

To start with, a 1-sphere packing $\mathcal S$ of $X_{3,3}^3$ is indicated in Figure 3 that contains in the fashion of Figure 2 eight $6\times 6$ arrays each standing for the disposition of vertices in an embedding of a copy of $\Pi_3^3$ in a torus.
In each such array, the black 6-tuples represent centers of 1-spheres in $\mathcal S$. There are two such centers in the first, (resp., third), [resp., fifth] row, namely in columns 1 and 4, (resp. 3 and 6), [resp., 5 and 2]. Each dark-gray 6-tuple stands for a vertex adjacent to one of the said 1-sphere centers located in a different copy of $\Pi_3^3$ via transposition $(\epsilon)=(34)$.
There are two of these dark-gray 6-tuples: in the second, (resp., fourth), [resp., sixth] row of each $6\times 6$ array, namely in columns 2 and 5, (resp., 4 and 1), [resp., 6 and 3].
This divides the black and dark-gray 6-tuples in each $6\times 6$ array into three $2\times 2$ sub-arrays obtained from the diagonal black 6-tuples by transpositions $(12)$ and $(56)$ and their composition. The left and center of Figure 4 represents, with the same 6-tuple shades of Figure 3, its upper-left copy of $\Pi_3^3$, namely $X(123)\square X(456)$.

\begin{center}TABLE I
$$\begin{array}{||c||cc|cc|cc||}%\hline
X(123)\square X(456) & 123456 & 213456 & 312564 & 132564 & 231645 & 321645\\
X(214)\square X(365) & 214365 & 124365 & 421653 & 241653 & 142536 & 412536\\
X(326)\square X(154) & 326154 & 236154 & 632541 & 362541 & 263415 & 623415\\
X(135)\square X(246) & 135246 & 315246 & 513462 & 153462 & 351624 & 531624\\
X(246)\square X(135) & 246135 & 426135 & 624351 & 264351 & 462513 & 642513\\
X(154)\square X(326) & 154326 & 514326 & 415263 & 145263 & 541632 & 451632\\
X(365)\square X(214) & 365214 & 635214 & 536142 & 356142 & 653421 & 563142\\
X(456)\square X(123) & 456123 & 546123 & 645231 & 465231 & 564312 & 645312\\%\hline
\end{array}$$\end{center}

\begin{center}TABLE II
$$\begin{array}{c|c|c|c|c|}\hline

^{X(162)\square X(534)}_{X(165)\square X(234)} &^{162534}_{165234}&^{612534}_{615234}&^{162543}_{165243}&^{612543}_{615243}\\\hline

^{X(163)\square X(425)}_{X(164)\square X(325)} &^{163425}_{164325}&^{613425}_{614325}&^{163452}_{164352}&^{613452}_{614352}\\\hline

^{X(256)\square X(134)}_{X(251)\square X(634)} &^{256134}_{251634}&^{526134}_{521634}&^{256143}_{523416}&^{526143}_{521643}\\\hline

^{X(254)\square X(316)}_{X(253)\square X(416)} &^{524361}_{523461}&^{254361}_{523461}&^{524316}_{523416}&^{254316}_{523416}\\\hline

^{X(431)\square X(652)}_{X(436)\square X(152)} &^{431652}_{436152}&^{341652}_{346152}&^{431625}_{436125}&^{341652}_{346152}\\\hline

^{X(432)\square X(516)}_{X(435)\square X(216)} &^{432516}_{435216}&^{342516}_{345216}&^{432561}_{435261}&^{342561}_{345261}\\\hline

\end{array}$$\end{center}

Table I lists on its leftmost column the copies of $\Pi_3^3$ of Figure 3, followed to their right by three pertaining pairs of 6-tuples encodable as $(a_{i,1},a_{i,2},a_{i,3})$, where $i\in I_8$.
For instance, $a_{1,1}=\{123456,213456\}$, $a_{1,2}=\{312564,132564\}$, etc. Consider the following pairs of pairs of black 6-tuples in the main diagonals of the eight $6\times 6$ arrays in Figure 3 related by the permutation $(12)(34)(56)$:

\begin{eqnarray}\label{J}^{\{a_{1,1},a_{2,1}\},\;\{a_{1,2},a_{4,1}\},\;\{a_{1,3},a_{3,1}\},\;\{a_{2,2},a_{5,1}\},\;\{a_{2,3},a_{6,2}\},\;\{a_{3,2},a_{7,1}\},}
_{\{a_{3,3},a_{5,2}\},\;\{a_{4,2},a_{6,1}\},\;\{a_{4,3},a_{7,2}\},\;\{a_{5,3},a_{8,2}\},\;\{a_{6,3},a_{8,1}\},\;\{a_{7,3},a_{8,3}\}.}\end{eqnarray}

The eight copies of $\Pi_3^3$ in Figure 3 induce a subgraph $X'_{3,3}$ of $X_{3,3}^3$  (right of Figure 4) whose vertex set
admits a partition into 48 1-spheres around the  black 6-tuples, with a partial total of
288 vertices. Moreover, $X'_{3,3}$ has an E-set $J$ formed by the black 6-tuples, encoded in the pairs of display (\ref{J}).
Consider the vertices of the remaining 12 copies of $\Pi_3^3$ in $X_{3,3}^3$ at distance 2 from a center of a 1-sphere among the cited 48. There are 192 such vertices, 16 in each of the 12 copies as the union of four copies of a product $J'\times J''$ of E-sets as in Section~\ref{spack} and inducing four 4-cycles in the copy.
The graph induced by the remaining 20 vertices in the copy contains four 1-spheres whose vertices via $F_\epsilon$ are centers of similar 1-spheres.
As a result we have the formation of double spheres, see below.
Table II allows to select 24 centers of pairwise disjoint 1-spheres to cover half of the resulting $240=12\times 20$ vertices: choose one 1-sphere center per pair of two 6-tuples in each box in the table. There are $144$ vertices in the 24 1-spheres. In sum, we obtain $\frac{3}{5}6!$ vertices of $X_{3,3}^3$ packed into $72=48+24$ 1-spheres.

Let us apply the definitions of double-sphere and $\mathcal S$-sphere in Subsection~\ref{T} with $X=X_{3,3}^3$ and $X'=X'_{3,3}$. By adding to each 1-sphere $\Sigma$ in the above packing of $X'$ the end-vertices of the $(\epsilon)$-colored edges departing from $\Sigma$, where $(\epsilon)=(34)$, a corresponding $\mathcal S$-sphere $\Sigma'$ is obtained enlarging $\Sigma$. On the other hand,
the 24 1-spheres selected above can be extended into 24 double-spheres, which forms a double-sphere packing.
A transformation of the 1-sphere packing ${\mathcal S}$ in Figure 3 into a perfect special (Subsection~\ref{T}) packing of $X_{r,t}^3$ is obtained by enlarging the 48 1-spheres that pack perfectly $X'$ into corresponding $\mathcal S$-spheres by adding the 192 vertices not in $X'$ and at distance 2 from the centers of the 48 1-spheres. The reader may compare this with the $\mathcal S$-sphere packing of $X_{2,2}^3$ suggested on the right of Figure 1.

Selecting instead 24 centers of 1-spheres to be the neighbors via the transposition $(23)$ (or $(13)$) of the 24 centers allowed above by means of Table II leaves room to selecting additional 24 centers of 1-spheres in the six still untouched copies of $\Pi_3^3$. The selection of the 24 new centers of 1-spheres in those six copies must be done via the transposition $(45)$ (or $(46)$). This yields
a packing of $X_{3,3}^3$ by 96 1-spheres comprising $576=\frac{4}{5}|V(X_{3,3}^3)|$ vertices. Observe that the 96 corresponding centers are obtained by modifying the original 1-sphere centers both adjacently and alternatively, idea to be generalized in Theorems~\ref{t5}.

\section{RENUMBERING TREE VERTICES}\label{s6}

In generalizing the maximum localized packing density of Section~\ref{s5}, we found it convenient to modify the order of vertices of the tree $\tau^3_{r,t}$ in items (i)-(ii) of Subsection~\ref{r10} by letting instead:
{\bf(i$'$)} $1$ and $r^*=r+1$ denote the vertices of respective degrees $r$ and $t$ in $\tau^3_{r,t}$ so that $\epsilon=1r^*$;
{\bf(ii$'$)} $2,\ldots,r$ (resp., $r^*+1,\ldots,n$) denote the vertices adjacent to vertex $1$ (resp., $r^*$) in $\tau^3_{r,t}$.

We exemplify this modification via Figure 5, on whose top a representation of the copy $X(12)\square X(34)$ of $\Pi_2^2$ is given that presents, before and after (symbol) $\square$, the copies of $K_2$ constituting $X(12)$ and $X(34)$, respectively. Similar representations can be given for $X(32)\square X(14)$, $X(14)\square X(32)$ and $X(34)\square X(12)$, forming with $X(12)\square X(34)$ a subgraph $X'_{2,2}$ of $X_{2,2}^3$ preceding the subgraph $X'_{3,3}$ of $X_{3,3}^3$ in Section~\ref{s5}. The two remaining squares $X(13)\square X(24)$ and $X(24)\square X(13)$ are shaded in light-gray color in Figure 1 (that used the original vertex numbering in items (i)-(ii), Subsection~\ref{r10}) and form a second subgraph $X''_{2,2}$ of $X_{2,2}^3$.

\begin{figure}[htp]
\hspace*{1.1cm}
\includegraphics[scale=0.27]{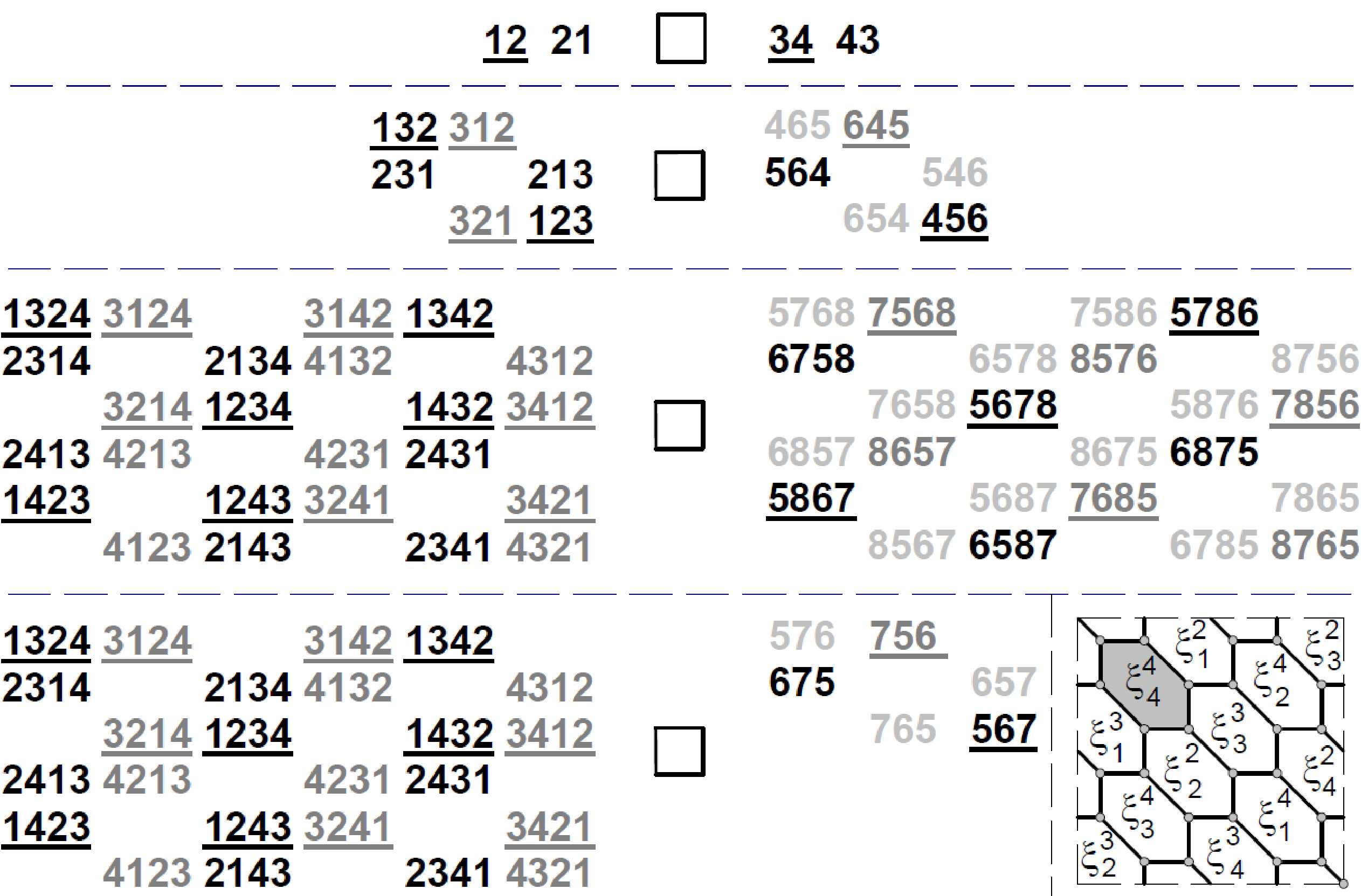}
\caption{Interpretations of $\Pi_2^2$, $\Pi_3^3$, $\Pi_4^4$ and $\Pi_4^3$}
\end{figure}

Subsequently in Figure 5, a similar representation of the cartesian product $X(123)\square X(456)$ is given that shows, before and after $\square$, the 6-cycles $X(123)$ and $X(456)$, respectively, by presenting adjacent vertices contiguously: horizontally, vertically and diagonally between upper-left and lower-right.
Here, the black centers of the three 1-spheres in the main diagonal of the $6\times 6$ array  representing $X(123)\square X(456)$ in Figure 3 (but with the vertex order assumed above in this section) are recovered by: {\bf(A)} taking a partition of $V(X(123))$ into the E-sets $\xi^1_1=1(23),\xi^1_2=2(13),\xi^1_3=3(12)$ (Subsection~\ref{xyz}) given by: {\bf(i)} underlined-black color for $\xi^1_1=\{123,132\}$, {\bf(ii)} (not underlined) black color for $\xi^1_2=\{213,231\}$ and {\bf(iii)} underlined-dark-gray color for $\xi^1_3=\{312,321\}$; {\bf(B)} assigning the three colors of (A) respectively to the even-parity vertices in $X(456)$ as follows: {\bf(i)} $456\in\xi^4_4$, {\bf(ii)} $564\in\xi^4_5$ and {\bf(iii)} $645\in\xi^4_6$, while the odd-parity vertices, namely $465$, $546$ and $654$, shown in light-gray, do not intervene; {\bf(C)} concatenating the vertices of $X(123)$ and $X(456)$ having a common color.

Now, we embed each copy of $X^2_4$ into a torus, as in the lower-right corner of Figure 5, with its copies $\xi^j_i$, ($j\in\{2,3,4\}$ ; $i\in I_4$), of $X^2_3$ presented as above into their places. This way, the previous representation of $X(123)\square X(456)$ is extended to $\Pi_4^4$ as in the lower two instances of Figure 5, where the shown cartesian products can be denoted $X(1234)\square X(5678)$ and $X(1234)\square X(567)$, this one obtained by restricting, i.e. puncturing $X(1234)\square X(5678)$.

In the third case of Figure 5, the coloring used for
$X(123)\square X(456)$ above is extended with a fourth color: (not underlined) dark-gray. On the left of $\square$, the colors correspond to the E-sets $\xi^1_i=i(I_4\setminus\{i\})$, where $i\in I_4$. On the right of $\square$, the even-parity $4$-tuples are given the same color $i$ when their intersection with an E-set of the partition $\{\xi^5_j;j=5,6,7,8\}$ starts with  $j=i+4$.
As mentioned, the situation for $X(1234)\square X(567)$ can be considered a restriction of that of $(1234)\square X(5678)$. We may write $X(567)=(567,\xi^7_7,657,\xi^6_5,756,\xi^7_6,576,\xi^6_7,675,$

\noindent $\xi^7_5,765,\xi^6_6)$.

In a typical cartesian product $\Pi_r^t=X_r^2\square X_t^2$, where $2<t\le r$, we notice that: {\bf(A)} the subset $Q$ of vertices of the copy $X(r^*\cdots n)$ of $X_t^2$, where $r^*=r+1$, which as $t$-tuples have the same parity as the $t$-tuple $r^*\cdots n$ has a partition into $t$ subsets $Q_i$ with the $t$-tuples in $Q_i$ starting at $(r+i)$, for every $i\in I_t$; {\bf(B)} the vertex set of the copy $X(1\cdots r)$ of $X_r^2$ has a partition into the $r$ E-sets $\xi^1_j$\, for every $j\in I_r$; {\bf(C)} it eases treatment to consider the $n$-tuples obtained by concatenating every $r$-tuple in $\xi^1_i$ with every $t$-tuple in $Q_i$\,, for every $i\in I_t$.

The convenience of the new vertex numbering is that to obtain a maximal number of disjoint 1-sphere centers in the copies of $\Pi_r^t$, say  $X(1\cdots r)\square X(r^*\cdots n)$, we can order both factors of these products in the same direction, resulting in transpositions between the first entry of either an initial $r$- or a terminal $t$-tuple with any of the remaining entries of that tuple, plus the transposition of both first entries. We concatenate initial $r$-tuples and terminal $t$-tuples whenever they have the same color (as in the instances of Figure 5), where the color set of the second factor in the product must coincide with, or be contained in, the color set of the first factor, considering that the second coloring here is given on the elements of the alternate subgroup $A_t\subset S_t$ while the first coloring is taken from a partition of $S_r$ into E-sets.

\section{NONUNIFORM SPHERE PACKING}\label{s8}

\begin{center}TABLE III
$$\begin{array}{||l||r|r|r|r|r|r||l||r|r|r|r|r|r||}
^{r\,;\,k}_{--}&^{\;\;\;0}_{--}&^{\;\;\;1}_{--}&^{\;\;\;2}_{--}&^{\;\;\;3}_{--}&^{\;\;\cdots}_{--}&^{\;\;\Sigma_r}_{--}&^{r\,;\,k}_{--}&^{\;\;\;0}_{--}&^{\;\;\;1}_{--}&^{\;\;\;2}_{--}&^{\;\;\;3}_{--}&^{\;\;\cdots}_{--}&^{\;\;\Sigma'_r}_{--}\\
^2_3&^4_8&^{\;\;2}_{12}&&&^{\cdots}_{\cdots}&^{\;\;6}_{20}&^2_3&^4_8&^{0}_{8}&&&^{\cdots}_{\cdots}&^{\;\;4}_{16}\\
^4_5&^{16}_{32}&^{\;\;\;48}_{\;160}&^{\;\;6}_{60}&&^{\cdots}_{\cdots}&^{\;\;70}_{252}&^4_5&^{16}_{32}&^{24}_{64}&^{\;\;0}_{24}&&^{\cdots}_{\cdots}&^{\;\;40}_{120}\\
^6_7     &^{\;\;64}_{128}&^{\;\;480}_{1344}&^{\;\;360}_{1680}&^{\;\;20}_{280}&^{\cdots}_{\cdots}&^{\;\;924}_{3432}&^6_7&^{\;\;64}_{128}&^{160}_{384}&^{120}_{480}&^{\;\;0}_{80}&^{\cdots}_{\cdots}&^{\;\;344}_{1072}\\
^{\cdots}&^{\cdots}&^{\cdots}&^{\cdots}&^{\cdots}   &^{\cdots}&^{\cdots}&^{\cdots}&^{\cdots}&^{\cdots}&^{\cdots}&^{\cdots}&^{\cdots}&^{\cdots}\\
\end{array}$$
\end{center}

Let $r>1$. If $z,z'\in I_n$ with $|z-z'|=r$, we denote $\mathbf{z}=\{z,z'\}$.
There are $2^r$ copies of $\Pi_r^r$ of the form $\Pi_r^r=X(a_1a_2\cdots a_r)\square X(a'_1$ $a'_2\cdots a'_r)$ with $\mathbf{a_i}=\{a_i,a'_i\}=\{i,r+i\}=\mathbf{i}$, for $i\in I_r$. The subgraph $X'_{r,r}$ induced by these copies possesses an E-set $J$ constructed as in Sections~\ref{s5}--\ref{s6}. Here, $J$ also dominates a subset
$\{y_1b_2\cdots b_ry_rd_2$ $\cdots d_r\,|\,\{b_2,\ldots,b_r\}=\{y'_1,y_2,\ldots y_{r-1}\};\{d_2,\ldots,d_r\}=\{y'_r,y'_2,\ldots,y'_{r-1}\}\}$ in each copy of
$\Pi_r^r$ of the form $\Pi_r^r=X(y_1y_1'y_2\cdots y_{r-1})\square$\\
\noindent$(y_ry'_ry'_2$ $\cdots y'_{r-1})$ in $X_{r,r}^3$ with $y_z\in\mathbf{y}_z$ for $z\in I_r$ and
$\{\mathbf{y}_z\,|\,z\in I_r\}=\{\mathbf{z}\,|\,z\in I_r\}$.
The ${2r\choose r}$ copies of $\Pi_r^r$ in $X_{r,r}^3$ are of the following types:
\begin{eqnarray}\label{dis}\begin{array}{rll}                                                                              ^{X(a_1a_2\cdots a_r)}&^{\square}&^{X(a'_1a'_2\cdots a'_r);}\\
^{X(a_1a'_1a_3a_4\cdots a_r)}&^{\square}&^{X(a_2a'_2a'_3a'_4\cdots a'_r);}\\
^{X(a_1a'_1a_2a'_2a_5a_6\cdots a_r)}_{\cdots}&^{\square}&^{X(a_3a'_3a_4a'_4a'_5a'_6\cdots a'_r);}_{\cdots}\\
^{X(a_1a'_1\cdots a_ka'_ka_{2k+1}a_{2k+2}\cdots a_r)}_{\cdots}&^{\square}&^{X(a_{k+1}a'_{k+1}\cdots a_{2k}a'_{2k}a'_{2k+1}a'_{2k+2}\cdots a'_r);}_{\cdots.}
\end{array}\end{eqnarray}
Let $X'_{r,r},X''_{r,r},X'''_{r,r}$, $\ldots$, $X^{(k^*)}_{r,r}, \ldots$ be the subgraphs induced respectively by the types in the first, second, third, $\ldots$, $k^*$-th, $\ldots$ lines of display (\ref{dis}), where $k^*=k+1$.
The number of times each $X^{(k^*)}_{r,r}$ occurs in $X_{r,r}^3$ is given by the
sequence A051288 \cite{oeis}, presentable as a number triangle $T$ each of whose terms $T(r,k)$, read by rows ($r\ge 0$; $k=0,1,\dots,\lfloor r/2\rfloor$), $T(r,k)$, is the number of paths of $r$ upsteps $U$ and $r$ downsteps $D$ with exactly $k$ subpaths $UUD$. In fact, $T(r,k)={r\choose 2k}2^{r-2k}{2k\choose k}.$
The left of Table III illustrates $T$, where each row of values $T(r,k)$ adds up to $\Sigma_r={2r\choose r}$. Note $F_\epsilon$ has edges only between contiguous subgraphs $X^{(k)}_{r,r}$ and $x^{(k^*)}_{r,r}$, for $k=0,1,\ldots,\lfloor r/2\rfloor$.

In continuation to our approach in Sections~\ref{s5}--\ref{s6}. the right of Table III (the sum of which rows is indicated by $\Sigma'_r$) gives $(r!)^{-2}$ times the number of vertices covered by a maximum $\alpha$-E-set $K$. The resulting quotient is denoted $S(r,k)$. Then, $S(r,k)\le T(r,k)$.
The intersection of such $K$ and each copy of $\Pi_r^r$ in $X^3_{r,r}-X'_{r,r}$ is a product of two E-sets of $X^2_r$ by an argument extending that of the last three paragraphs of Section~\ref{s5} that departs from the vertices in $X''$ at distance 2 from the E-set $J$ constructed in $X'$.
In fact, a copy $\Pi'$ of $P_r^r$ in $X''_{r,r}$ and an $\alpha$-E-set extending $J$ intersect at most in a product $J'\times J''$ of E-sets. We take the vertices of such $J'\times J''$ as centers of 1-spheres in $\Pi'$. These centers may appear in pairs of adjacent vertices in $X_{r,r}^3$ yielding a packing ${\mathcal S}''$ by double-spheres whose centers form a subset $J^*$. By displacing
the vertices of $J^*$ via alternate adjacency in the two components $X_r^2$ of
each copy of $\Pi_r^r$ in $X''_{r,r}$, we
replace ${\mathcal S}''$ by a 1-sphere packing ${\mathcal S}'$ containing
$(2r)\times((r-1)!)^2$ vertices of the $(r!)^2$ vertices of each copy of
$\Pi_r^r$ in $X''_{r,r}$\,, a proportion of $2/r$ of the vertices of $X''_{r,r}$.
The same proportion is kept in the remaining $X''',\ldots,X^{(k^*)},\ldots$, starting by choosing 1-spheres in the copies of $\Pi_r^r$ in $X'''_{r,r}$ avoiding the neighbors (via $F_\epsilon$) of the 1-spheres in ${\mathcal S}'$ and then using ``exact'' paths in Johnson graphs as in Section~\ref{s2}.

\begin{theorem}\label{t5}
If $n=2r>4$, where $r\in\mathbb{Z}$, then: {\bf(a)}
a connected subgraph $X'_{r,r}$ induced in $X_{r,r}^3$ by the disjoint union of $2^r$ copies of $\Pi_r^r$ has a perfect 1-sphere packing $\mathcal S$;
{\bf(b)} $\mathcal S$ cannot be extended to a perfect 1-sphere packing of $X_{r,r}^3$; {\bf(c)} a maximum nonuniform 1-sphere packing ${\mathcal S}'$ of $X_{r,r}^3$ is obtained as an extension of $\mathcal S$ that yields an $\alpha$-E-set of $X_{r,r}^3$ with $\alpha=\Sigma'_r/\Sigma_r=(2^r+\frac{2}{r}P_r)/{2r\choose r}$, where
$P_r={2r\choose r}-2^r$ if $r$ is odd and $P_r={2r\choose r}-2^r-{r\choose r/2}$ if r is even; {\bf(d)} $\frac{n}{r^2}<\alpha<1$.
\end{theorem}

\begin{proof}
Apart from the $2^r$ copies of $\Pi_r^r$ in $X'_{r,r}$ there are in $X_{r,r}^3$:
${2r\choose r} - 2^r$ copies of $\Pi_r^r$ if $r$ is odd and
${2r\choose r} - 2^r - {r\choose r/2}$ copies of $\Pi_r^r$ if $r$ is even. In these copies we could select products $\Upsilon=a(b_2\cdots b_r).a'(c_2\cdots c_r)$ formed by E-sets $a(b_2\cdots b_r)$ and $a'(c_2\cdots c_r)$. The cardinality of each such $\Upsilon$ is $((r-1)!)^2$, its vertices as centers of 1-spheres pairwise disjoint in their copies of $\Pi_r^r$ but for $F_\epsilon$ possibly allowing the formation of pairwise disjoint double-spheres instead.
As in the final discussion in Section~\ref{s5} (presented with our initial notation, as in Table III), we could displace adjacently and alternatively the 1-sphere centers in the first and second components $X_r^2$ of $\Pi_r^r$. This can modify those double 1-spheres into pairwise disjoint 1-spheres which cover at best $2r((r-1)!)^2$ vertices of $X_{r,r}^3$.
The number of times that $(r!)^2$ appears at most in the vertex counting of the
resulting nonuniform packing of $X_{r,r}^3$ is
$2^r+2rP_r((r-1)!)^2/(r!)^2= 2^r+2rP_r/r^2=2^r+\frac{2}{r}P_r$. Thus,
an $\alpha$-E-set of $X_{r,r}^3$ has $\frac{n}{r^2}<\alpha\le(2^r+\frac{2}{r}P_r)/{2r\choose r}$.
This value of $\alpha$ is an $\alpha<1$.
\end{proof}

As in the bottom example of Figure 5, the general case of $X_{r,t}^3$ with $r\ge t$ can be considered a restriction, if necessary,
of the one of $X_{r,r}^3$ by means of the puncturing technique mentioned in Section~\ref{s6}. This way, we get the following.

\begin{cor}\label{r6} Let $r>t>1$. A maximum nonuniform 1-sphere packing of $X_{r,t}^3$ exists that yields an $\alpha$-E-set of $X_{r,t}^3$ with $\frac{n}{rt}<\alpha\le\frac{\Sigma'_t}{\Sigma_r}<1$, where $\Sigma'_t=(2^t+\frac{2}{t}P_t)$ and $\Sigma_r={2r\choose r}$ with
$P_t={2t\choose t}-2^t$ if $t$ is odd and $P_t={2t\choose t}-2^t-{t\choose t/2}$ if t is even.
\end{cor}

\end{document}